\documentclass[12pt]{amsart}
\usepackage[margin=1in]{geometry} 
\usepackage{amsmath,amsthm,amssymb,amsfonts}
\usepackage{parskip}
\usepackage{algpseudocode}
\usepackage{mathtools}
\usepackage{enumitem}
\usepackage{bm}
\usepackage{tikz}
\usepackage{tikz-cd, graphicx}
\usepackage{fancyvrb, tgcursor}
\usepackage{biblatex}
\DeclarePairedDelimiter\ceil{\lceil}{\rceil}
\DeclarePairedDelimiter\floor{\lfloor}{\rfloor}

\newcommand{\p}[1]{\left(#1\right)}
\newcommand{\s}[1]{\left[#1\right]}
\newcommand{\set}[1]{\left\{#1\right\}}

\newcommand{\pres}[1]{\left\langle#1\right\rangle}
\newcommand{\inv}{^{-1}}

\newcommand{\bigmid}{\ \middle|\ }
 
\newcommand{\E}{\mathbb{E}}

\newcommand{\ca}{\mathcal}

\newcommand{\bbr}{\mathbb{R}}

\newcommand{\bbz}{\mathbb{Z}}
\newcommand{\bbp}{\mathbb{P}}

\newcommand{\tmix}{t_{\text{mix}}}

\DeclareMathOperator{\Var}{Var}

\DeclareMathSymbol{\derange}{\mathord}{operators}{"3C}

\makeatletter
\newcommand{\tpitchfork}{%
  \vbox{
    \baselineskip\z@skip
    \lineskip-.52ex
    \lineskiplimit\maxdimen
    \m@th
    \ialign{##\crcr\hidewidth\smash{$-$}\hidewidth\crcr$\pitchfork$\crcr}
  }%
}
\newtheorem*{rep@theorem}{\rep@title}
\newcommand{\newreptheorem}[2]{%
\newenvironment{rep#1}[1]{%
 \def\rep@title{#2 \ref{##1}}%
 \begin{rep@theorem}}%
 {\end{rep@theorem}}}
\makeatother

\newtheorem{thm}{Theorem}[section]
\newtheorem{prop}[thm]{Proposition}
\newtheorem{cor}[thm]{Corollary}

\newtheorem{lem}[thm]{Lemma}
\newtheorem*{prop*}{Proposition}
\newtheorem*{cor*}{Corollary}
\newtheorem*{lem*}{Lemma}
\newtheorem*{thm*}{Theorem}

\theoremstyle{definition}

\newtheorem*{claim*}{Claim}

\newtheorem*{conj*}{Conjecture}

\theoremstyle{definition}
\newtheorem*{defn*}{Definition}

\newtheorem*{rmk*}{Remark}

\newreptheorem{lem}{Lemma}
\newreptheorem{thm}{Theorem}

\begin{document}

\title{Mixing Time and Cutoff for the k-SEP}
\author{Eyob Tsegaye}
\address{Department of Mathematics, Princeton University, Princeton, NJ 08540}
\email{eyob@princeton.edu}
\maketitle


\begin{abstract}
    We investigate the mixing time of the capacity $k$ simple exclusion process (also called the partial exclusion process) of Schultz and Sandow with $m$ particles on a segment of length $N$. We show that the $k$-SEP exhibits cutoff at time $\frac{1}{2k\pi^2}N^2\log m$. We also introduce a related complete multi-species process that we call the $S_{k,N}$ shuffle and show that this process exhibits cutoff at time $\frac{1}{2k\pi^2}N^2\log (kN)$. This extends the celebrated result of Lacoin, which determined the mixing time of the symmetric simple exclusion process on a segment of length $N$ and the adjacent transposition shuffle, and proved cutoff in both.
\end{abstract}


\section{Introduction}

The exclusion process is a widely studied model of interacting particles on a finite or infinite graph. In the process, particles move from site to site in the graph along the edges with pre-determined rates, but there are fermionic interactions between the particles, so that no two particles can occupy the same site. That is, if one particle tries to move to a vertex that is occupied, it bounces back and returns to the vertex it was previously occupying. In 1994, Schutz and Sandow \cite{schutz-sandow} introduced a generalization of the exclusion process which they called partial exclusion. In this setting, each site can hold at most $k \in \bbz_+$ particles, and the particle jump rate from site $i$ to site $j$ is proportional to $a_i(k-a_j)$ where $a_i$ denotes the number of particles at site $i$. Since then, a wealth of study has gone into establishing duality in this model and variations of it (\cite{cggr-13}\cite{frs-20}\cite{cgrs-14}\cite{acr-20}\cite{cgr-19}\cite{gkrv-09}) as well as studying limits, equilibria, and fluctuations (\cite{am-23}\cite{fgjs-23}\cite{frs-21}\cite{cs-21}\cite{cgr-20}). Throughout the literature, the name of this model varies, with some calling it the partial exclusion process and some calling it the $k$-SEP (or $k$-ASEP if appropriate). We will use the latter to emphasize the dependence of our results on $k$. We should also mention another generalization of the exclusion process to bounded capacity proposed by Keisling in 1998 \cite{keisling} and studied further by Seppalainen \cite{seppalainen}. In this model, the particle jump rate from site $i$ to site $j$ is constant so long as $i$ is not empty and $j$ is not at capacity.

In this paper, we study the total variation mixing time of the $k$-SEP model of Schutz and Sandow on the path of length $N$ with $m$ particles. We establish cutoff for the model at time $\frac{1}{2k\pi^2}N^2\log m$. Like most proofs of total variation cutoff, this requires a detailed analysis of the chain.

We now formally define the $k$-SEP. Given positive integers $k,N$, and $m\leq kN$, let 
$$\Omega_{k,N,m} := \set{\gamma: [N] \to \{0, \ldots, k\} \bigmid \sum_{x=1}^N \gamma(x) = m}, $$
where we use the notation $[n] := \{1, \ldots, n\}$. The function $\gamma$ tells us how many particles are at each site and thus we call $\gamma$ a configuration.

The process on this state space is as follows. For each $x \in [N]$, define $\tau_x^+: \Omega_{k,N,m} \to \Omega_{k,N,m}$ to be the transformation that moves a particle at site $x$ to the right (if possible), and let $\tau_x^-$ be the transformation that moves a particle at site $x$ to the left (if possible). That is, if $\gamma' = \tau_x^+(\gamma)$, then
$$ \gamma'(y) = \begin{cases}
    \gamma(x)-1 & y=x < N,\ \gamma(x+1)<k \\
    \gamma(x+1)+1 & y=x+1 \leq N,\ \gamma(x+1) < k \\
    \gamma(y) & \text{otherwise.}
\end{cases} $$
The transformation $\tau_x^-$ is defined analogously. The generator $\ca{L}$ of the chain is given by 
$$ (\ca{L}_{\text{kSEP}}f)(\gamma) = \sum_{x=1}^{N} \s{p_x\p{f(\tau_x^+ \gamma)-f(\gamma)}+q_x\p{f(\tau_x^- \gamma)-f(\gamma)}}$$
where the rates of transition, $p_x$ and $q_x$, are defined as
$$ p_x = \gamma(x)(k-\gamma(x+1)), \qquad q_x = \gamma(x)(k-\gamma(x-1)). $$

\begin{figure}
    \centering
    \includegraphics[scale=0.5]{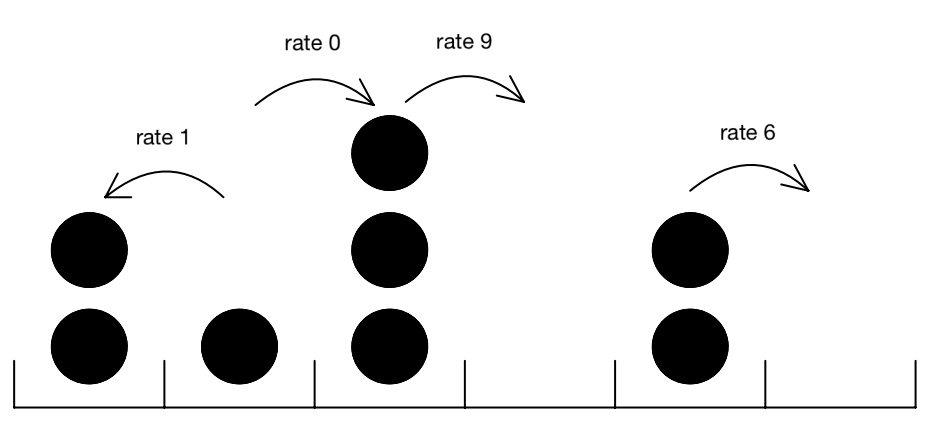}
    \caption{One particular configuration of the $k$-SEP with state space $\Omega_{3,6,8}$. The right-pointing arrows have the rates $p_x$ written above them while the left-pointing arrows have $q_x$.}
    \label{fig:ksep-chain}
\end{figure}

In words, each particle at site $x$ moves to the left with rate $k-\gamma(x-1)$ and to the right with rate $k-\gamma(x+1)$ (see Figure \ref{fig:ksep-chain}). This chain is reversible with respect to the distribution given by
$$\mu^{N,k,m}(\gamma) = \binom{Nk}{m}\inv \prod_{x=1}^N \binom{k}{\gamma(x)}.$$
Note that when $k=1$, this is just the classical symmetric simple exclusion process.

Our first result is that cutoff for the $k$-SEP occurs at time $\frac{1}{2k\pi^2}N^2\log m$. More formally, let $P^{k,N,m}_t$ be the law of the chain associated with the generator $\ca{L}_{\text{kSEP}}$. We define
\begin{align*}
   d^{k,N,m}(t) &:= \max_x \|P^{k,N,m}_t(x, \cdot) - \mu^{k,N,m}\|_{TV} = \max_{x, A} |P_t^{k,N,m}(x, A)-\mu^{k,N,m}(A)| \\
   \tmix^{k,N,m}(\epsilon) &:= \inf_{t\geq 0} \{d^{k,N,m}(t) \leq \epsilon\}.
\end{align*}
Our result is then:
\begin{thm}\label{ksep-cutoff-thm}
    Suppose $m = m(N) \to \infty$ and $m \leq kN/2$. For all $\epsilon \in (0,1)$, we have
    $$\lim_{N\to\infty} \frac{\tmix^{k,N,m}(\epsilon)}{\frac{1}{2k\pi^2}N^2\log m} = 1.$$
\end{thm}

As the limiting behavior of $\tmix^{k,N,m}(\epsilon)$ is independent of $\epsilon$, this is precisely saying that cutoff occurs in this process. For an overview of cutoff, see \cite{levin-peres-wilmer} chapter 18 and \cite{diaconis-cutoff}.

In studying exclusion processes, and interacting particle systems more generally, one typically has the particles be indistinguishable. However, there are a variety of ``multi-species" models which distinguish different classes of particles (see e.g. \cite{casini-multispecies}, \cite{aggarwal-colored} for examples and more references). In the case where each particle is distinguished (we call this ``complete" multi-species) and when taken on a finite set of sites, the model often corresponds to some natural card shuffle. For example, the complete multi-species SEP on a path would correspond to the adjacent transposition shuffle while the complete multi-species SEP on a star graph would correspond to the star transposition shuffle. We could even make these shuffles asymmetric in a similar sense to exclusion processes and analyze those, see for example \cite{labbe-lacoin}, \cite{lingfu}. Often these card shuffles shed light on the dynamics of the exclusion processes they are related to, and of course the card shuffles are interesting in their own right. With this motivation, we introduce a complete multi-species analog of the $k$-SEP which we call the $S_{k,N}$ shuffle.

Fix a positive integer $k$. Let $\sigma: [N] \to \ca{P}([kN])$ be a map sending each element to a disjoint subset of size $k$, i.e. the inverse of a $k$-to-1 map $[kN] \to [N]$. If $k=1$, this is just a permutation, and we interpret this permutation as a deck of cards where $\sigma(x)$ gives the value of the card at position $x$ in the deck. Therefore, for general $k$, we call this a $k$-permutation and interpret this as a ``deck of packets," where each packet contains $k$ cards in an indistinguishable order. Notationally, we write this set as $S_{k,N}$, so that $S_{1,N} = S_N$, the symmetric group on $N$ elements. Furthermore, for $\sigma \in S_{k,N}$, we let $\sigma(x)_i$ denote the $i$th smallest number in $\sigma(x)$.

During the shuffle, each pair of cards in adjacent packets swaps which packets they are in with rate 1. More formally, for each $x \in [N-1]$ and $i,j \in [k]$, let $\tau_x^{i,j}: S_{k,N} \to S_{k,N}$ swap the $i$th smallest element of $\sigma(x)$ with the $j$th smallest element of $\sigma(x+1)$. That is, if $\sigma(x) = \{a_1,\ldots, a_k\}$ with $a_1 < \ldots < a_k$ and $\sigma(x+1) = \{b_1, \ldots, b_k\}$ with $b_1 < \ldots < b_k$, then $\sigma' := \tau_x^{i,j}(\sigma)$ is defined by $\sigma'(x) = \sigma(x) \cup \{b_j\} - \{a_i\}$ and $\sigma'(x+1) = \sigma(x+1) \cup \{a_i\} - \{b_j\}$, with $\sigma'(y) = \sigma(y)$ for all $y \neq x,x+1$. The generator of this chain is given by
$$ (\ca{L}_{\text{shuffle}} f)(\sigma) = \sum_{x=1}^{N-1}\sum_{i,j=1}^k \s{f(\tau_x^{i,j} \sigma)-f(\sigma)}, $$
so that each transformation $\tau_x^{i,j}$ occurs with rate 1 (see Figure \ref{fig:skn-shuffle}). This chain is reversible with respect to the uniform distribution $\mu^{N,k}(\sigma) = (k!)^N/(kN)!$.

\begin{figure}
    \centering
    \includegraphics[scale=0.5]{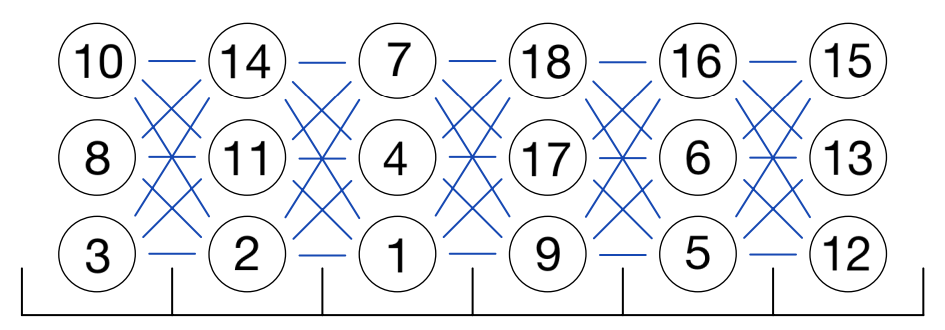}
    \caption{An example of an element $\sigma \in S_{3,6}$ (i.e. a 3-permutation). For example, $\sigma(3) = \{1,4,7\}$. Blue lines indicate rate 1 swaps in the $S_{k,N}$ shuffle. }
    \label{fig:skn-shuffle}
\end{figure}

Our second result is that cutoff for the $S_{k,N}$ shuffle occurs at time $\frac{1}{2k\pi^2}N^2\log kN$. To write this formally, let $P^{k,N}_t$ be the law of the chain associated with the generator $\ca{L}_{\text{shuffle}}$ and define $d^{k,N}(t)$ and $\tmix^{k,N}(\epsilon)$ analogously to $d^{k,N,m}(t)$ and $\tmix^{k,N,m}(\epsilon)$.

\begin{thm}\label{shuffle-cutoff-thm}
    For all $\epsilon \in (0,1)$, we have
    $$\lim_{N\to\infty} \frac{\tmix^{k,N}(\epsilon)}{\frac{1}{2k\pi^2}N^2\log(Nk)} = 1.$$
\end{thm}

Again, since the limiting behavior of $\tmix^{k,N}(\epsilon)$ is independent of $\epsilon$, this is precisely saying that cutoff occurs in the $S_{k,N}$ shuffle.

We briefly mention the wonderful result of Chen and Saloff-Coste \cite{chen-saloffcoste} which tells us that the discrete versions of these chains also exhibit cutoff and where their mixing time is.

Our methods are motivated by a celebrated result of Lacoin \cite{lacoin-at} which showed both that the symmetric exclusion process with $m$ particles on an interval of length $N$ exhibits cutoff at time $\frac{1}{2\pi^2}N^2\log m$ and that the adjacent transposition shuffle on $N$ cards exhibits cutoff at time $\frac{1}{2\pi^2}N^2\log N$. This used a novel approach of combining coupling arguments with a censoring inequality, the latter of which had been introduced by Peres and Winkler in the context of monotone spin systems \cite{peres-winkler}. Since then, this use of censoring and coupling arguments has been very fruitful in solving many other mixing time problems. See for example, the $p$-biased adjacent transposition shuffle \cite{labbe-lacoin}, the cyclic adjacent transposition shuffle \cite{nam-nestoridi}, the simple exclusion process with open boundaries \cite{gantert}, and the $S_k$ shuffle block dynamics \cite{priestley}.

We adapt the methods of \cite{lacoin-at} to this problem and establish cutoff for the $k$-SEP and the $S_{k,N}$ shuffle. The methods do not trivialize the problem, as the state space and transition rates are genuinely different, and for example lead to a slightly more delicate area decay argument in the upper bound for the $k$-SEP. One of the biggest differences, which will become clearer, is that because $S_{k,N}$ is not a group, we do not have a bijection between $k$-permutations of a prescribed semi-skeleton and some subgroup of $S_{Nk}$. However, lemma \ref{messy-action-lem} lets us maneuver around this.

We now provide an outline of the paper. In order to prove theorems \ref{ksep-cutoff-thm} and \ref{shuffle-cutoff-thm}, we first prove a lower bound for the mixing time of the $k$-SEP and the $S_{k,N}$ shuffle in section \ref{lower-bd-section}. This is done by explicitly computing eigenvalues of the $k$-SEP chain and using a continuous version of Wilson's method (see e.g. \cite{levin-peres-wilmer}, section 13.5) given in \cite{priestley}. The lower bound for the $S_{k,N}$ shuffle follows immediately by taking $m=\floor{kN/2}$. Afterwards, in sections 4 and 5 we develop machinery similar to \cite{lacoin-at} that will help in proving upper bounds for the two models. Finally, sections 6 and 7 will prove the upper bounds for the $S_{k,N}$ shuffle and $k$-SEP, respectively, using aforementioned techniques of combining censoring and coupling.


\section{Notation and Preliminaries}\label{preliminaries-section}

In this section we introduce notation that will be used throughout the paper and establish preliminary facts about standard objects. First let us make clear how the two chains introduced are related. We have a projection of Markov chains (in the sense of \cite{levin-peres-wilmer}, section 2.3.1) from the $S_{k,N}$ shuffle to the $k$-SEP via the map $\varphi_m: S_{k,N} \to \Omega_{k,N,m}$ by $\varphi_m(\sigma)(x) = |\sigma(x) \cap [m]|$. This projection will be very useful in giving us a quick way to relate properties between the two chains.

It will sometimes be useful to think of $k$-permutations as a collapsed version of permutations in $S_{Nk}$. In particular, define the map $\theta: S_{Nk} \to S_{k,N}$ by $(\theta(\sigma))(x) = \{\sigma((x-1)k+1), \ldots, \sigma(xk)\}$. This is a $(k!)^N$-to-1 surjective map onto $S_{k,N}$.

Often we will deal with a graphical interpretation of $k$-permutations and $k$-SEP configurations rather than the objects themselves. Define the height function $\tilde{\sigma}$ of $\sigma \in S_{k,N}$ by
$$ \tilde{\sigma}(x,y) = \sum_{z=1}^x |\sigma(z)\cap [y]| - \frac{xy}{N} $$
and similarly define the height function $\eta$ of $\gamma \in \Omega_{k,N,m}$ by
$$ \eta(x) = \sum_{z=1}^x |\gamma(z)\cap [m]| - \frac{xm}{N}. $$
\begin{figure}
    \centering
    \includegraphics[scale=0.5]{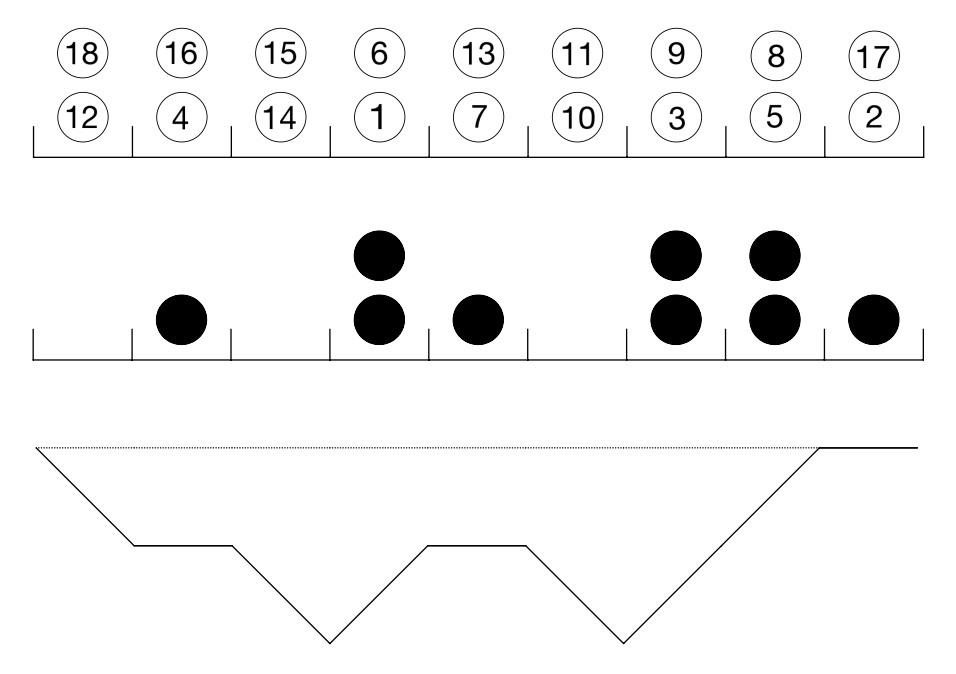}
    \caption{An example showcasing the projection $\varphi$. The top image is an element $\sigma \in S_{2,9}$, the middle is its projection $\varphi_9(\sigma) \in \Omega_{9,2,9}$, and the bottom is the height function $\eta$ associated to $\varphi_9(\sigma)$ (which is also $\tilde{\sigma}(\cdot, 9)$).}
    \label{fig:heightfn}
\end{figure}
We may also think of these as lattice paths for some particular lattice (see e.g. Figure \ref{fig:heightfn}). Note that these maps are injective, in particular
\begin{align*}
    1_{y\in\sigma(x)} &= \tilde{\sigma}(x,y)-\tilde{\sigma}(x-1,y)-\tilde{\sigma}(x,y-1)+\tilde{\sigma}(x-1,y-1) + \frac{1}{n} \\
    \gamma(x) &= \eta(x)-\eta(x-1)+\frac{m}{N}.
\end{align*}
We order the height functions by their partial ordering as real functions, that is $\tilde{\sigma} \leq \tilde{\tau}$ iff $\tilde{\sigma}(x,y) \leq \tilde{\tau}(x,y)$ for all pairs $x,y$. Since the map $\sigma \mapsto \tilde{\sigma}$ is injective, this induces a partial ordering on $S_{k,N}$. In $S_{k,N}$, we let $\mathbf{1}$ denote the maximal element, which is equal to $\theta(id)$ where $id$ is the identity permutation in $S_{Nk}$.

In $\Omega_{k,N,m}$, write $\land$ for the lattice path corresponding to the configuration where all particles are pushed to the left and $\lor$ for the lattice path corresponding to the configuration where all particles are pushed to the right (the symbols are motivated by how the lattice paths look). These correspond to the maximal and minimal configurations in $\Omega_{N,k,m}$. We can explicitly compute $\land(x) = \min(xk, m) - \frac{xm}{N}$.

For the rest of the paper, if not specified, limits are taken as $N\to\infty$ (this applies to little $o$ notation as well). When we say ``with high probability" or ``whp," we mean with probability tending to 1 as $N\to\infty$. Finally, for a Poisson process $\ca{T} = (\ca{T}_n)_{n\geq 0}$, we may say that a ``clock rings" at time $t$ if $t=\ca{T}_n$ for some $n$.


\section{Lower bound}\label{lower-bd-section}

In this section we show the following.

\begin{thm}\label{lower-bd-thm}
    There is a constant $C$ such that for any $\epsilon > 0$ and $m \leq Nk/2$,
    $$\tmix^{N,k,m}(\epsilon) \geq \frac{1}{2k\pi^2}N^2\log m+\frac{CN^2}{k}\log\p{\frac{1}{1-\epsilon}}+o(N^2).$$
\end{thm}

We make a quick note that because the $k$-SEP is a projection of the $S_{k,N}$ shuffle, we have that $\tmix^{k,N}(\epsilon) \geq \tmix^{k,N,m}(\epsilon)$ for all $m\leq Nk/2$. In particular, once we prove the above, we will immediately have the following lower bound for the $S_{k,N}$ shuffle mixing time.

\begin{cor}
    There is a constant $C$ such that for any $\epsilon > 0$
    $$\tmix^{N,k}(\epsilon) \geq \frac{1}{2k\pi^2}N^2\log(kN)+\frac{CN^2}{k}\log\p{\frac{1}{1-\epsilon}} + o(N^2).$$
\end{cor}

Our main tool in proving theorem \ref{lower-bd-thm} will be the following adaptation of Wilson's lemma for continuous time processes.

\begin{lem} [\cite{priestley}, lemma 4.1] \label{wilsons-lem}
    Let $\Psi$ be an eigenfunction of the generator of a continuous Markov chain $X_t$ with eigenvalue $-\lambda$. If $\Var(\Psi(X_t)) \leq R$ for all $t\geq 0$, then for all $\epsilon > 0$,
    $$ \tmix(1-\epsilon) \geq \frac{1}{\lambda}\log(\|\Psi\|_\infty) - \frac{1}{2\lambda}\log(8R/\epsilon). $$
\end{lem}

Thus the goal will first be to gather eigenfunctions of the $k$-SEP. Luckily by first ``looking up" to the $S_{k,N}$ shuffle and then ``looking down" by only following a single card, we can get a good number of eigenfunctions.

Suppose $\gamma\in \Omega_{N,k,m}$ and let $\eta$ be the associated height function. Define the functions $f_j: \Omega_{N,k,m} \to \bbr$ by
$$ f_j(\gamma) = \sum_{x=1}^N \gamma(x)\cos\p{\frac{(2x-1)j\pi}{2N}}, \qquad \qquad f_j(\eta) = \sum_{x=1}^N \eta(x)\sin\p{\frac{xj\pi}{N}}. $$

\begin{lem}
    If $0 \leq j \leq N-1$, then $f_j$ defined as above is an eigenfunction of the $k$-SEP chain with eigenvalue $-\lambda_{N,k,j}$ where $\lambda_{N,k,j} = 2k(1-\cos(j\pi/N)) = k\frac{j^2\pi^2}{N^2}(1+o(1))$ (as $N\to\infty$).
\end{lem}

Now we are ready to prove theorem \ref{lower-bd-thm}.

\begin{proof}[Proof of Theorem \ref{lower-bd-thm}]
    Writing $\eta_t$ for the Markov process on the height functions, the fact that $f_1$ is an eigenvector gives us that $M_s := e^{(s-t)\lambda_N}f_1(\eta_s)$ is a martingale for any $t$. By construction $M_t = f_1(\eta_t)$. Therefore, if $\eta_0 = \land$, then $\Var(f_1(\eta^\land_t)) = \Var(M_t) = \E\s{\pres{M}^2_s}$.
    
    Note that each time a Poisson clock rings, $f_1$ changes by at most a constant $C$. Furthermore, a change will occur with rate at most $2km$ (each of the $m$ particles hops left or right with rate $\leq 2k$). Therefore, $\E\s{\pres{M}^2_s} \leq Cmk\int_0^t e^{2(s-t)\lambda_{N,k}}dt \leq CmN^2$.

    Finally, we see that $f_1(\land) = \sum_{x=1}^N \land(x)\sin(\frac{x\pi}{N}) \geq \frac{1}{2}\sum_{x=N/4}^{3N/4}\land(x) \geq \frac{1}{2}\sum_{x=N/4}^{3N/4} m/4 = mN/16$. The second inequality follows from the expression for $\land(x)$. Thus, $\|f_1\|_\infty \geq mN/16$.

    
    The lower bound in theorem \ref{lower-bd-thm} thus follows when taking $\Psi = f_1$, $R = CmN^2$, and $\lambda = 2k(1-\cos(\pi/N)) = (1+o(1))k\pi^2/N^2$ in lemma \ref{wilsons-lem}.
\end{proof}


\section{Grand Coupling and Properties}

\subsection{Construction}\label{grand-coupling-section}

Our goal for this section is to construct a grand coupling of the $S_{k,N}$ shuffle which preserves the order on $S_{k,N}$. To this end, with each $x \in \{1,\ldots, N-1\}$ we associate $k^2$ independent Poisson processes $\ca{T}^{x,i,j} = (\ca{T}_n^{x,i,j})_{n\geq 0}$ where $i,j \in [k]$ each with rate 2. Let $(U_n^{x,i,j})_{n \geq 0}$ be an independent iid field of Bernoulli variables with parameter 1/2. At time $t=\ca{T}_n^{x,i,j}$ (for $n\geq 1$), we update $\sigma_t$ via one of two actions depending on the value of $U_n^{x,i,j}$.

\begin{itemize}
    \item If $U_n^{x,i,j} = 1$ and $\sigma_{t^-}(x+1)_j < \sigma_{t^-}(x)_i$ or $U_n^{x,i,j} = 0$ and $\sigma_{t^-}(x+1)_j > \sigma_{t^-}(x)_i$, then $\sigma_t = \tau_x^{i,j}(\sigma_t)$
    \item Otherwise we do nothing.
\end{itemize}

Putting this action into words, we pick out the $i$th smallest card in packet $x$ and the $j$th smallest card in packet $x+1$. The Bernoulli variable $U$ will tell us whether to sort or reverse sort these cards by transposing them (that is, by taking the cards out of their packets and putting them in the other packet).

We highlight what this does to the height function $\tilde{\sigma}$. Note that the only values in $\{\tilde{\sigma}(z, y)\}$ that change are those for which $z=x$. Suppose $|\sigma(x) \cap [y]| = a$ and $|\sigma(x+1) \cap [y]| = b$ and $t = \ca{T}_n^{x,i,j}$. If $U_n^{x,i,j} = 1$ and $i > a$ and $j \leq b$, then $\tilde{\sigma}_t(x, y) = \tilde{\sigma}_{t^-}(x,y)+1$. If $U_n^{x,i,j} = 0$ and $i \leq a$ and $j > b$, then $\tilde{\sigma}_t(x,y) = \tilde{\sigma}_{t^-}(x,y)-1$. Otherwise, nothing happens and $\tilde{\sigma}_t(x,y) = \tilde{\sigma}_{t^-}(x,y)$.

\subsection{Monotonicity of the Coupling}

We now verify that this coupling does in fact preserve the ordering given in section \ref{preliminaries-section}. That is, if $\xi \in S_{k,N}$ and we write $\sigma_t^\xi$ for the dynamics starting at $\xi$, then the following holds.

\begin{lem}\label{grand-monotonicity}
    If $\xi \geq \xi'$, then $\sigma_t^\xi \geq \sigma_t^{\xi'}$ for all $t\geq 0$.
\end{lem}

\begin{proof}
    This is proven inductively and follows \cite{lacoin-at}, proposition 3.1. Assume that $t = \ca{T}_n^{x,i,j}$. Suppose $U_n^{x,i,j} = 1$. We only have to prove that $\tilde{\sigma}_t^\xi(x,y) \geq \tilde{\sigma}_t^{\xi'}(x,y)$ for all $y$, since all other values stay fixed.

    Note that from our description of how the grand coupling affects the height function, it is impossible for $\tilde{\sigma}_t^\xi(x,y)$ to decrease while $\tilde{\sigma}_t^{\xi'}(x,y)$ increases. Therefore, if $\tilde{\sigma}_t^\xi(x,y) < \tilde{\sigma}_t^{\xi'}(x,y)$, then we must have that $\tilde{\sigma}_{t^-}^\xi(x,y) = \tilde{\sigma}_{t^-}^{\xi'}(x,y)$. However, we also know that $\tilde{\sigma}_{t^-}^\xi(x-1,y) \geq \tilde{\sigma}_{t^-}^{\xi'}(x-1,y)$ and $\tilde{\sigma}_{t^-}^\xi(x+1,y) \geq \tilde{\sigma}_{t^-}^{\xi'}(x+1,y)$. Therefore, if we say that $|\sigma_{t^-}^\xi(x) \cap [y]| = a$, $|\sigma_{t^-}^{\xi'}(x) \cap [y]| = a'$, $|\sigma_{t^-}^\xi(x+1) \cap [y]| = b$, and $|\sigma_{t^-}^{\xi'}(x+1) \cap [y]| = b'$, then $a' \geq a$ and $b \geq b'$. As we saw in the grand coupling, the only way $\tilde{\sigma}_{t^-}^{\xi'}(x,y)$ can increase is if $U_n^{x,i,j}=1$ and $i > a'$ and $j \leq b'$. However, this also implies that $i > a$ and $j \leq b$, and therefore $\tilde{\sigma}_{t^-}^\xi(x,y)$ also increases. Similarly, if $\tilde{\sigma}_{t^-}^{\xi'}(x,y)$ decreases, then $\tilde{\sigma}_{t^-}^\xi(x,y)$ also decreases. In particular, monotonicity of the chain is preserved.
\end{proof}

\subsection{FKG and Censoring Inequalities}

We have a useful FKG-type inequality for $k$-permutations. Here, an increasing function on $S_{k,N}$ means that $f(\sigma) \geq f(\tau)$ if $\sigma \geq \tau$ in $S_{k,N}$.

\begin{thm} [\cite{lacoin-at}, proposition 3.5]
    For any pair of increasing functions $f,g: S_{k,N} \to \bbr$, we have that $\mu(f\cdot g) \geq \mu(f)\mu(g)$, where $\mu$ is the uniform measure on $S_{k,N}$.
\end{thm}

We would also like to introduce a censoring inequality which will help us prove an upper bound on the mixing time. The general idea is that censoring updates can only make the system mix slower. This is a concept, first formalized in \cite{peres-winkler}, that is intuitive, but false for general Markov chains (see \cite{holroyd}).

We define a \textbf{censoring scheme} as a c\`adl\`ag function $\ca{C}: \bbr^+ \to \ca{P}([N-1] \times [k] \times [k])$. Given a censoring scheme $\ca{C}$ define the censored dynamics as those introduced in section \ref{grand-coupling-section}, but if $t = \ca{T}_n^{x,i,j}$, then we only update $\sigma_t$ if $(x, i, j) \in \ca{C}(t)$. Given a distribution $\nu$ on $S_{k,N}$, we let $P_t^{\nu, \ca{C}}$ denote the law of the censored dynamics at time $t$ with initial distribution $\nu$.

\begin{thm}\label{censoring-thm}
    If the probability distribution $\nu$ is an increasing function of $S_{k,N}$ and $\ca{C}$ is a censoring scheme, then
    $$ \|P_t^{\nu,\ca{C}}-\mu\| \leq \|P_t^\nu-\mu\|. $$
\end{thm}

This theorem is proved in several parts.

Let $\sigma_{x,i,j}^+$ be the $k$-permutation $\sigma$ after the $i$th smallest card in packet $x$ and the $j$th smallest card in packet $x+1$ are sorted. Let $\sigma_{x,i,j}^-$ be the same but with the cards reverse sorted. The distribution of $\nu$ after the clock $\ca{T}_{x,i,j}$ rings is precisely $\nu_x^{i,j}(\sigma) = (\nu(\sigma_{x,i,j}^+)+\nu(\sigma_{x,i,j}^-))/2$.

\begin{lem}
    Let $\nu$ be a probability distribution on $S_{k,N}$. If $\nu$ is increasing, then $\nu_x^{i,j}$ is increasing and $\nu \succeq \nu_x^{i,j}$.
\end{lem}

\begin{proof}
    If $\sigma \geq \xi$, then the proof of lemma \ref{grand-monotonicity} shows that $\sigma_{x,i,j}^+ \geq \xi_{x,i,j}^+$ and $\sigma_{x,i,j}^- \geq \xi_{x,i,j}^-$. Therefore, $\nu_x^{i,j}(\sigma) \geq \nu_x^{i,j}(\xi)$ since $\nu$ is increasing. The rest of the proof follows precisely as in \cite{lacoin-at}, lemma A.2.    
\end{proof}

\begin{lem}[\cite{lacoin-at}, Lemma A.4]
    Let $\nu$ and $\nu'$ be probability distributions on $S_{k,N}$. If $\nu$ is increasing and $\nu \preceq \nu'$, then $\|\nu-\mu\| \leq \|\nu'-\mu\|$.
\end{lem}

\begin{lem}
    Let $\nu_0$ be an increasing probability distribution on $S_{k,N}$ and fix a deterministic sequence $(x_1, i_1, j_1), \ldots, (x_m,i_m,j_m)$ with entries in $[N-1]\times [k]\times[k]$. Let $\nu$ be the distribution obtained from $\nu_0$ after successive updates at $(x_1,i_1,j_1),\ldots,(x_m,i_m,j_m)$ and let $\nu'$ be the same but with finitely many updates omitted. Then $\nu \preceq \nu'$.
\end{lem}

\begin{proof}
    The proof is identical to that of \cite{lacoin-at}, lemma A.5.
\end{proof}

From the last lemma, theorem \ref{censoring-thm} follows exactly as in \cite{lacoin-at}, section A.2.

\subsection{Monotonicity of Projections}

In what follows, we sometimes want to lose a little bit of information but still have monotonicity be preserved. Let $R\in \bbz_+$ be arbitrary. Let $x_j = \ceil{jN/R}$ and $y_j = k\ceil{jN/R}$. Given $\sigma \in S_{k,N}$, define the semi-skeleton $\widehat{\sigma}: \{0,\ldots, N\} \times \{0,\ldots, R\} \to \bbr$ by $\widehat{\sigma}(x, j) = \tilde{\sigma}(x, y_j)$. We define the skeleton $\bar{\sigma}: \{0,\ldots,R\}^2 \to \bbr$ by $\bar{\sigma}(i,j) = \tilde{\sigma}(x_i,y_j)$. We call $\widehat{S}_{k,N}$ the set of admissible semi-skeletons and $\bar{S}_{k,N}$ the set of admissible skeletons (the images of $S_{k,N}$ under the respective maps).

Before getting into the subject of this section, we first establish a key property relating the semi-skeleton back to the k-permutations with that semi-skeleton.

First we remark that the information contained in the semi-skeleton $\widehat{\sigma}$ is exactly the value of the sets $\sigma\inv(\{y_{i-1}+1, \ldots, y_i\})$. Keeping the analogy from the introduction about $\sigma$ encoding a deck of packets of cards, $\widehat{\sigma}$ tells us how many cards of each suit each packet contains (when the deck is in order, each packet only contains one kind of suit, though several packets may contain only spades for example). Therefore, the missing information is: among each suit, which cards are where (see Figure \ref{fig:semiskel}).

\begin{figure}
    \centering
    \includegraphics[scale=0.5]{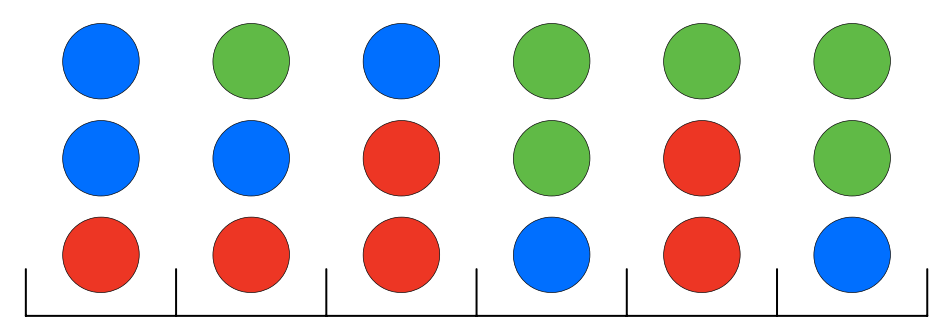}
    \caption{The information obtained from the semi-skeleton associated to the element $\sigma \in S_{3,6}$ given in Figure \ref{fig:skn-shuffle} with $R = 3$. In the image, red represents the numbers $\{1,\ldots, 6\}$, blue represents $\{7,\ldots, 12\}$, and green represents $\{13,\ldots,18\}$.}
    \label{fig:semiskel}
\end{figure}

Define an action of $S_{Nk}$ on $S_{k,N}$ as follows. For $\omega \in S_{k,N}$ we can think of its most ordered pre-image $\sigma_{\omega} \in S_{Nk}$ under $\theta$ (the ordering on $S_{Nk}$ given by that on $S_{Nk,1}$). Thus, given $\sigma \in S_{Nk}$, we define its action on $\omega$ by $\sigma \cdot \omega = \theta(\sigma \circ \sigma_{\omega})$ (see Figure \ref{fig:sigma-action} for an example). Note that while we suggestively call this an action, it is not a group action. However, this action allows us to pick out $k$-permutations which have the same semi-skeleton as follows.

\begin{figure}
    \centering
    \includegraphics[scale=0.5]{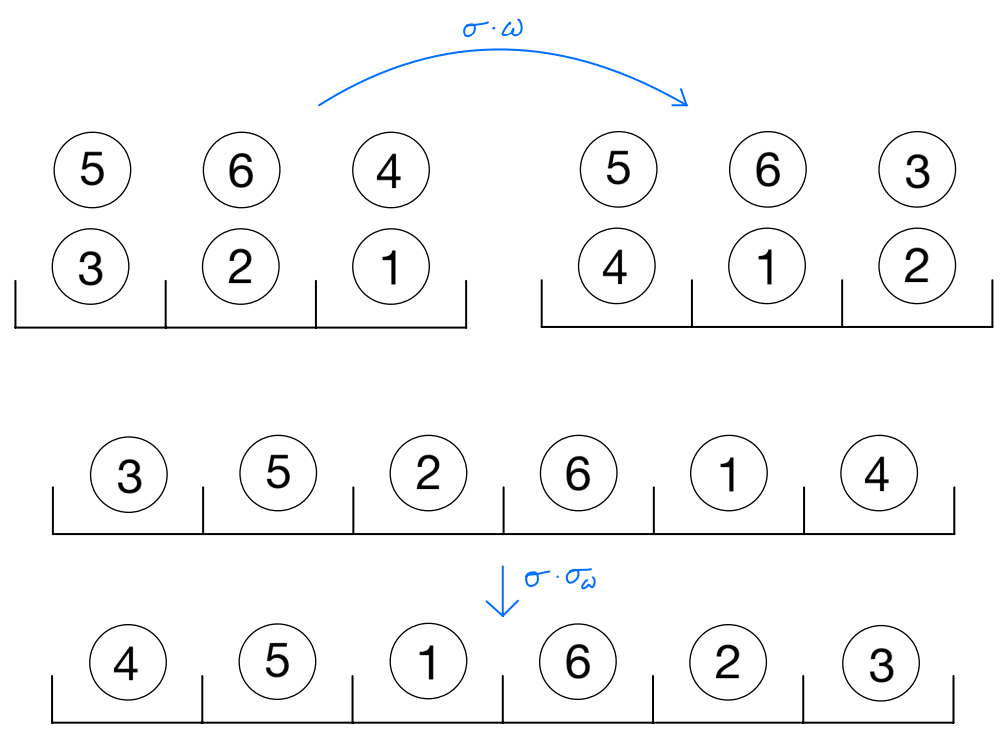}
    \caption{Example of the action $\sigma\cdot \omega$ with $R=3$. In the image $\omega \in S_{2,3}$ is taken to have $\omega(1) = \{3,5\}, \omega(2) = \{2, 6\}, \omega(3) = \{1,4\}$. Further we take $\sigma = (1\ 2)(3\ 4) \in \tilde{S}$. The top process is the action $\sigma \cdot \omega$ itself while the bottom is the underlying action $\sigma \circ \sigma_\omega$.}
    \label{fig:sigma-action}
\end{figure}

\begin{lem}\label{messy-action-lem}
    Let $\tilde{S} \cong \bigotimes_{i=1}^R S_{\Delta y_i}$ be the largest subgroup of $S_{Nk}$ which fixes the sets $\{y_{i-1}+1,\ldots,y_i\}$ for all $i\geq 1$. Also, let $G \cong S_k^N$ be the largest subgroup of $S_{Nk}$ which fixes the sets $\{xk+1,\ldots, (x+1)k\}$ for all $x \geq 0$. For any $k$-permutations $\omega$ and $\omega'$, we have that
    $$ |\sigma \in \tilde{S} : \sigma \cdot \omega = \omega'| = \begin{cases}
        0 & \text{if } \widehat{\omega}' \neq \widehat{\omega} \\
        |\tilde{S} \cap \sigma_\omega G\sigma_\omega\inv| & \text{if } \widehat{\omega}' = \widehat{\omega}.
    \end{cases} $$
\end{lem}

\begin{proof}
    First note that $\widehat{\sigma \cdot \omega} = \widehat{\omega}$ for every $\sigma \in \tilde{S}$ because of how we defined $\tilde{S}$, so for the rest of the proof, we can assume $\widehat{\omega}' = \widehat{\omega}$. The lemma requires some understanding of the action $\sigma \cdot \omega$ and thus $\theta$. If we have $\sigma_1,\sigma_2 \in S_{Nk}$ such that $\theta(\sigma_1) = \theta(\sigma_2)$, then by definition $\sigma_1(\{xk+1,\ldots, (x+1)k\}) = \sigma_2(\{xk+1,\ldots,(x+1)k\})$ for all $0\leq x \leq N-1$. However, this means that $\sigma_1\inv\sigma_2$ leaves the sets $\{xk+1,\ldots,(x+1)k\}$ invariant so that $\sigma_1G = \sigma_2G$. Conversely, it is straightforward that $\theta$ is constant on cosets of $G$. Therefore, $\theta$ gives a bijection between the cosets of $G$ and $S_{k,N}$.
    
    Now, we want to produce some $\sigma \in \tilde{S}$ such that $\sigma \cdot \omega = \omega'$ i.e. such that $\theta(\sigma \circ \sigma_\omega) = \omega'$. From our analysis, we know there is a unique coset $\sigma'G$ such that $\theta(\sigma'G) = \omega'$. We can let $\sigma_{\omega'}$ be the coset representative here, and thus we want to find $\sigma \in \tilde{S}$ such that $\sigma\sigma_\omega \in \sigma_{\omega'}G$. Therefore, we can alternatively express $\{\sigma \in \tilde{S} : \sigma \cdot \omega = \omega'\} = \tilde{S} \cap \sigma_{\omega'}G\sigma_{\omega}\inv$.
    
    Our next claim is that $\sigma_{\omega'}\sigma_{\omega}\inv \in \tilde{S}$. Indeed, since $\widehat{\omega}' = \widehat{\omega}$, we have that $\widehat{\sigma}_\omega(xk, i) = \widehat{\omega}(x,i) = \widehat{\omega}'(x,i) = \widehat{\sigma}_{\omega'}(xk,i)$ for all $x$ and $i$. But because of how we picked $\sigma_\omega$ and $\sigma_\omega'$, this last fact implies that $\widehat{\sigma}_\omega(x, i) = \widehat{\sigma}_{\omega'}(x, i)$ for all $x$ and $i$. Now from each semi-skeleton we can get the information $|\sigma(x) \cap \{y_{i-1}+1,\ldots, y_i\}|$, and thus $\sigma_{\omega'}\inv(\{y_{i-1}+1,\ldots, y_i\}) = \sigma_\omega\inv(\{y_{i-1}+1,\ldots, y_i\})$ for all $i$.
    
    Therefore, $\tilde{S} \cap \sigma_{\omega'}G\sigma_\omega\inv$ is just the coset $(\sigma_{\omega'}\sigma_\omega\inv)(\tilde{S}\cap \sigma_\omega G \sigma_\omega\inv)$ and it is now immediate that $|\sigma \in \tilde{S} : \sigma \cdot \omega = \omega'| = |\tilde{S} \cap \sigma_{\omega'}G\sigma_\omega\inv| = |\tilde{S}\cap\sigma_{\omega}G\sigma_\omega\inv|$ for every $\omega'$ such that $\widehat{\omega}' = \widehat{\omega}$.
\end{proof}

Now proceeding to the subject of this section, we equip $\bar{S}_{k,N}$ and $\widehat{S}_{k,N}$ with the point-wise partial ordering for real functions (for $f,g: X \to \bbr$, we say $f \geq g$ if $f(x) \geq g(x)$ for all $x$).

Given a probability measure $\nu$ on $S_{k,N}$, we write $\widehat{\nu}$ and $\bar{\nu}$ for the projections on $\widehat{S}_{k,N}$ and $\bar{S}_{k,N}$ respectively. Also, we write $\bar{\nu}_{i,j}$ for the image measure on the projection $\sigma \mapsto \bar{\sigma}(i,j)$.

\begin{thm}
    We have the following.
    \begin{enumerate}[label=(\roman*)]
        \item If $\bar{\sigma}^1 \geq \bar{\sigma}^2$ in $S_{k,N}$, then $\mu(\cdot \mid \bar{\sigma} = \bar{\sigma}^1) \succeq \mu(\cdot \mid \bar{\sigma} = \bar{\sigma}^2)$.
        \item If $(i,j) \in \{0,\ldots,R\}^2$ and $z_1 \geq z_2$ are two admissible values for $\bar{\sigma}(i,j)$, then we have $\mu(\cdot \mid \bar{\sigma}(i,j) = z_1) \succeq \mu(\cdot \mid \bar{\sigma}(i,j) = z_2)$.
        \item If $\nu$ is an increasing probability measure on $S_{k,N}$, then the density $\bar{\nu}/\bar{\mu}$ is an increasing function on $\bar{S}_{k,N}$.
        \item If $\nu$ is an increasing probability measure on $S_{k,N}$, then $\bar{\nu}_{i,j}/\bar{\mu}_{i,j}$ is an increasing function on the set of admissible values for $\bar{\sigma}(i,j)$.
    \end{enumerate}
\end{thm}

Note that the latter two items are just a product of integrating the first two items against the increasing function $\nu/\mu$. We prove (i) and remark that essentially the same proof works for (ii). To this end, we introduce two lemmas.

\begin{lem}\label{monotone-lem-1}
    If $\xi_1 \geq \xi_2$ in $\widehat{S}_{k,N}$ then $\mu(\cdot \mid \widehat{\sigma} = \xi_1) \succeq \mu(\cdot \mid \widehat{\sigma} = \xi_2)$.
\end{lem}

\begin{proof}
    We will prove this by giving an explicit coupling, our motivation being lemma \ref{messy-action-lem}. Indeed, pick $\rho \in \tilde{S}$ uniformly at random. Then for any $\xi$ pick $\sigma \in S_{k,N}$ such that $\widehat{\sigma} = \xi$ and such that $\sigma\inv(y_{i-1}+1,\ldots,y_i)$ is in order for all $i$ (that is, $\sigma\inv(a) \leq \sigma\inv(b)$ if $y_{i-1} < a < b \leq y_i$). Then we let $\sigma^\rho_{\xi} = \rho \cdot \sigma$. By lemma \ref{messy-action-lem}, $\sigma^\rho_\xi$ is uniform in $\mu(\cdot \mid \widehat{\sigma} = \xi)$ as desired.

    Secondly we would like to show that $\sigma^\rho_{\xi_1} \geq \sigma^\rho_{\xi_2}$. For convenience, write $\sigma_1 = \sigma^\rho_{\xi_1}$ and $\sigma_2 = \sigma^\rho_{\xi'}$. What we need to check is that $\tilde{\sigma}_1(x, y) \geq \tilde{\sigma}_2(x,y)$, or equivalently, $\sum_{z=1}^x |\sigma_1(z) \cap [y]| \geq \sum_{z=1}^x |\sigma_2(z) \cap [y]|$ for all $x,y$. WLOG let $y_{j-1} < y \leq y_j$. Let $A_i = \sum_{z=1}^{x} |\sigma_i(z) \cap [y_{j-1}]|$, let $B_i = \sum_{z=1}^x |\sigma_i(z) \cap \{y_{j-1}+1, \ldots, y_j\}|$, and let $C_i = \sum_{z=1}^x |\sigma_i(z) \cap \{y_{j-1}+1, \ldots, y\}|$. Framed in this way, we must check that $A_1+C_1 \geq A_2+C_2$, that is, that $A_1-A_2 \geq C_2-C_1$. Since $\xi_1 \geq \xi_2$, we know that $A_1 \geq A_2$ and $A_1+B_1 \geq A_2+B_2$, so that $A_1-A_2 \geq 0$ and $A_1-A_2 \geq B_2-B_1$.

    From our coupling, we may recall that $\tilde{S} \cong \bigotimes_{i=1}^R S_{\Delta y_i}$ and therefore write $\rho = (\rho_1, \ldots, \rho_R)$. Noting that $(\sigma_{\xi_i}^\rho)\inv(\{y_{j-1}+1,\ldots, y_j\})$ is governed precisely by $\rho_j$, we can express $C_i$ as $\sum_{z=1}^{B_i} 1_{\rho_j(z) \leq y} = |\rho_j([B_i]) \cap [y]|$. Therefore, $C_2-C_1 \leq 0 \leq A_1-A_2$ if $B_1 \geq B_2$, and otherwise $C_2-C_1 = |\sigma_j(\{B_1+1, \ldots, B_2\}) \cap [y]| \leq B_2-B_1 \leq A_1-A_2$. So indeed, $\sigma_\xi \geq \sigma_{\xi'}$.
\end{proof}

\begin{lem}\label{monotone-lem-2}
    If $\bar{\sigma}^1 \geq \bar{\sigma}^2$ in $\bar{S}_{k,N}$ then $\widehat{\mu}(\cdot \mid \bar{\sigma} = \bar{\sigma}^1) \succeq \widehat{\mu}(\cdot \mid \bar{\sigma} = \bar{\sigma}^2)$.
\end{lem}

\begin{proof}
    The proof is nearly identical to that of \cite{lacoin-at}, equation (A.10).
\end{proof}

With these two lemmas, the proof of the theorem is rather straightforward.

\begin{proof}
    Let $f$ be an increasing function on $S_{k,N}$. We define its projection $\widehat{f}$ on $\widehat{S}_{k,N}$ by averaging over the relevant $k$-permutations, i.e. $\widehat{f}(\xi) = \mu(f(\sigma) \mid \widehat{\sigma} = \xi)$. By lemma \ref{monotone-lem-1}, we know that $\widehat{f}$ is an increasing function on $\widehat{S}_{k,N}$. Finally, if $\bar{\sigma}^1 \geq \bar{\sigma}^2$ in $\bar{S}_{k,N}$, then by lemma \ref{monotone-lem-2}, we have $\mu(f(\sigma) \mid \bar{\sigma} = \bar{\sigma}^1) = \widehat{\mu}(\widehat{f}(\xi) \mid \bar{\xi} = \bar{\sigma}^1) \geq \widehat{\mu}(\widehat{f}(\xi) \mid \bar{\xi} = \bar{\sigma}^2) = \mu(f(\sigma) \mid \bar{\sigma} = \bar{\sigma}^2)$.
\end{proof}


\section{Tools for the Upper Bound}

\subsection{Heat Equation}

Let us compute the expected drift of the height function $\partial_t \E[\tilde{\sigma}_t(x,y)]$. Suppose that $|\sigma_{t^-}(x) \cap [y]| = a$ and $|\sigma_{t^-}(x+1) \cap [y]| = b$. Now suppose $t = \ca{T}^{x,i,j}$ for some $i,j$ (this is the only case where $\tilde{\sigma}(x,y)$ will change). If $(M_1, M_2) = (|\sigma_{t}(x) \cap [y]|, |\sigma_{t}(x+1) \cap [y]|)$ then we have
$$ (M_1, M_2) = \begin{cases}
    (a+1, b-1) & \text{if } i<a, j\leq b, U^{x,i,j} = 1 \\
    (a-1, b+1) & \text{if } i\leq a, j > b, U^{x,i,j} = 0 \\
    (a,b) & \text{otherwise.}
\end{cases} $$

\begin{figure}
    \centering
    \includegraphics[scale=0.5]{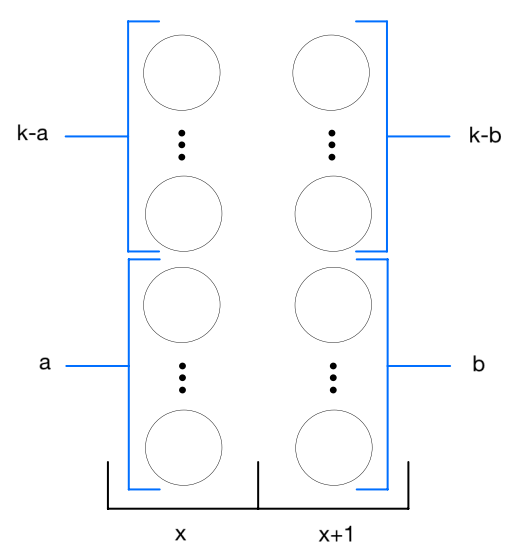}
    \caption{If one of the clocks at site $x$ rings, $|\sigma(x) \cap [y]|$ will change only if an element $\leq y$ is picked at one of the sites $\{x,x+1\}$, and an element $> y$ is picked at the other.}
    \label{fig:heat}
\end{figure}

See Figure \ref{fig:heat} for clarification. Thus, $M_1 = \tilde{\sigma}_t(x,y)-\tilde{\sigma}_t(x-1,y)+\frac{y}{N}$ increases by 1 with rate $b(k-a)$ and decreases by 1 with rate $a(k-b)$. In particular, if we let $f = \E[\tilde{\sigma}_t(x,y)]$, then $\partial_t f = b(k-a)-a(k-b) = k(b-a)$. However, $b$ is precisely $\tilde{\sigma}_t(x+1,y)-\tilde{\sigma}_t(x,y)+\frac{y}{N}$ while $a$ is $\tilde{\sigma}_t(x,y)-\tilde{\sigma}_t(x-1,y) + \frac{y}{N}$. Therefore, $\partial_t f = k\E[\tilde{\sigma}_t(x+1,y)-2\tilde{\sigma}_t(x,y)+\tilde{\sigma}_t(x-1,y)] = k\Delta_x f$ where $\Delta_x g = g(x+1)-2g(x)+g(x-1)$ is the discrete Laplacian acting on the $x$ coordinate.

By fixing $y$ and considering the average height function $f(x,t) = \E[\tilde{\sigma}_t(x,y)]$ as only a function of $x$ and $t$, this means that it actually satisfies the following heat equation:
$$ \begin{cases}
    \partial_t f = k\Delta_x f & \text{for } x=1,\ldots,N-1 \text{ and } t>0 \\
    f(0, t) = f(N, t) = 0 & \text{for } t\geq 0 \\
    f(x, 0) = \tilde{\sigma}_0(x,y) & \text{for } x=0,\ldots,N.
\end{cases} $$

We thus get the following result.

\begin{prop}\label{heat-prop}
    For all $\sigma_0 \in S_{k,N}$ and $t \geq 0$, we have
    $$ \max_x \E[\tilde{\sigma}_t(x,y)] \leq 4\min(y, kN-y)e^{-\lambda_{k,N}t}$$
    where $\lambda_{k,N} = 2k(1-\cos(\pi/N)) = k\pi^2/N^2 (1+o(1))$ (as $N \to \infty$). In particular,
    $$ \max_{x,y} \E[\tilde{\sigma}_t(x,y)] \leq 2kNe^{-\lambda_{k,N}t}. $$
    And further, when $\sigma_0 = \mathbf{1}$, we have
    $$ \E[\tilde{\sigma}_t(x,y)] \geq \frac{\min(y, kN-y)}{\pi}\sin\p{\frac{x\pi}{N}}e^{-\lambda_{k,N}t}. $$
\end{prop}

\begin{proof}
    The proof is almost identical to that of \cite{lacoin-at}, lemma 4.1.
\end{proof}

\subsection{Rough Upper Bound}

This is the first tool for which we use the setting of the $k$-SEP rather than the $S_{k,N}$ shuffle. Let $k$ be fixed. For all $N$ sufficiently large and $m \in \{0,\ldots,kN\}$, we have

\begin{lem}\label{wilson-upper-bound-lem}
    $d^{N,k,m}(t) \leq 10me^{-t\lambda_{k,N}}$ where $\lambda_{k,N} = 2k\p{1-\cos\p{\frac{\pi}{N}}}$.
\end{lem}

\begin{proof}
    This proof essentially follows that of proposition 6.5 of \cite{lacoin-at} which follows that of section 8.1 in \cite{wilson-lozenge}, thus we simply highlight the main differences from \cite{lacoin-at}. To bound $d^{N,k,m}(t)$, we construct a coupling of $\sigma_t^{\xi}$ and $\sigma_t^{\xi'}$ (for any $\sigma,\sigma' \in S_{k,N}$) and show that $\bbp[\exists i \in [m]: (\sigma_t^{\xi})\inv(i) \neq (\sigma_t^{\xi'})\inv(i)] \leq 10me^{-t\lambda_{k,N}}$. The coupling is defined as follows:
    \begin{itemize}
        \item If $(\sigma_t^\xi(x)))_i \notin \sigma_t^{\xi'}(x)$ and $(\sigma_t^\xi(x+1))_j \notin \sigma_t^{\xi'}(x+1)$, then the transition corresponding to $\tau_x^{i,j}$ occurs independently with rate one for each of the two processes.
        \item Otherwise, let $i'$ and $j'$ be such that $(\sigma_t^{\xi'}(x))_{i'}= (\sigma_t^{\xi}(x))_i$ and $(\sigma_t^{\xi'}(x+1))_{j'}= (\sigma_t^{\xi}(x))_j$ (with $i'$ or $j'$ being arbitrary in $[k]$ if their definition does not make sense). Simultaneously, the transition corresponding to $\tau_x^{i,j}$ happens in $\sigma_t^\xi$ while the transition corresponding to $\tau_x^{i',j'}$ happens in $\sigma_t^{\xi'}$.
    \end{itemize}
    We may let $X_t^i = (\sigma_t^\xi)\inv(i)$ and $Y_t^i = (\sigma_t^\xi)\inv(i)$. Then $(X_t^i, Y_t^i)$ is a Markov chain with the transitions described in the proof of proposition 6.5 of \cite{lacoin-at} concerning the same quantities, but with transition rate $k$ rather than rate one. By union bound, we have that $\bbp[\exists i \in [m]: (\sigma_t^{\xi})\inv(i) \neq (\sigma_t^{\xi'})\inv(i)] \leq m\max_{x,y}\bbp_{x,y}[X_t \neq Y_t]$ where $(X_t,Y_t)$ is the Markov chain starting at $(x,y)$ with the same transitions as $(X_t^i,Y_t^i)$. However, as we said this is just a rate $k$ rather than a rate one version of the process described in \cite{lacoin-at}, by writing the latter process as $(X'_t, Y'_t)$, we have that $\bbp_{x,y}[X_t \neq Y_t] = \bbp_{x,y}[X'_{kt} \neq Y'_{kt}]$. By \cite{lacoin-at}, lemma 6.6, this last quantity is bounded by $10\exp(-t\lambda_{k,N})$, completing the proof.
\end{proof}

\begin{cor}\label{wilson-upper-bound-cor}
    $d^{N,k}(t) \leq 10kNe^{-t\lambda_{k,N}}$.
\end{cor}

\subsection{Decomposing the Mixing Procedure}

Suppose we have some measure $\nu$ on $S_{k,N}$. We define the measure $\tilde{\nu}$ on $S_{k,N}$ by $\tilde{\nu}(\sigma) = \frac{1}{|S|}\widehat{\nu}(\widehat{\sigma})$ where $S = \{\sigma' \in S_{k,N} \mid \widehat{\sigma}' = \widehat{\sigma}\}$. Using lemma \ref{messy-action-lem}, we can reframe this.

\begin{lem}
    Let $\tilde{S}$ be the largest subgroup of $S_{Nk}$ which fixes the sets $\{y_{i-1}+1,\ldots,y_i\}$ for all $i\geq 1$. This is isomorphic to $\bigotimes_{i=1}^R S_{\Delta y_i}$. For any $k$-permutation $\omega$, we have that
    $$ \tilde{\nu}(\omega) = \frac{1}{|\tilde{S}|}\sum_{\sigma \in \tilde{S}} \nu(\sigma \cdot \omega). $$
\end{lem}

\begin{proof}
Because of the definition of $\tilde{S}$, we know that $\widehat{\sigma \cdot \omega} = \widehat{\omega}$ for all $\sigma \in \tilde{S}$. Thus, we want to say that each $k$-permutation $\omega'$ such that $\widehat{\omega}' = \widehat{\omega}$ has the same value for $n(\omega') := |\{\sigma \in \tilde{S} \mid \sigma \cdot \omega = \omega'\}|$, as this would show that each $\omega'$ gets an equal share in the sum. But this is just a consequence of lemma \ref{messy-action-lem}.
\end{proof}

Finally, we have the following result.

\begin{lem}[\cite{lacoin-at}, Lemma 4.3]\label{mix-decompose-lemma}
    For all probability measures $\nu$ on $S_{k,N}$ we have $\|\tilde{\nu}-\mu\| = \|\widehat{\nu}-\widehat{\mu}\|$ and thus
    $$ \|\nu-\mu\| \leq \|\nu-\tilde{\nu}\| + \|\widehat{\nu}-\widehat{\mu}\|. $$
\end{lem}


\section{Upper Bound for $S_{k,N}$}\label{skn-upper-bound-section}

We would like to show that $\tmix^{N,k}(\epsilon) \leq (1+o(1))\frac{1}{2k\pi^2}N^2\log kN$. Therefore, we need to show the following.

\begin{thm}\label{skn-upper-bound-thm}
Fix $\epsilon, \delta > 0$. For $N$ sufficiently large,
$$ d_{N,k}\p{(1+\delta)\frac{1}{2k\pi^2}N^2\log kN} \leq \epsilon $$
\end{thm}

Our main broad tool will be lemma \ref{mix-decompose-lemma}. Following \cite{lacoin-at}, we let
\begin{align*}
    t_1 &= (\delta/3)\frac{N^2}{2k\pi^2}\log kN \\
    t_2 &= (1+2\delta/3)\frac{N^2}{2k\pi^2}\log kN \\
    t_3 &= (1+\delta)\frac{N^2}{2k\pi^2}\log kN
\end{align*}

and proceed as follows.
\begin{enumerate}
    \item From time 0 to $t_1$, we run the grand coupling dynamics, while censoring the transitions between packets $x_i$ and $x_i+1$ for $i=1,\ldots, R-1$ with $R := \ceil{1/\delta}$. We use corollary \ref{wilson-upper-bound-cor} to show that the cards with labels in $\{y_i+1, \ldots, y_{i+1}\}$ are mixed for each $i=0, \ldots, R-1$.
    \item From time $t_1$ to $t_2$, we run the non-censored dynamics and show that the skeleton is close to equilibrium afterwards.
    \item From time $t_2$ to $t_3$, we run the censored dynamics again and show that the semi-skeleton is close to equilibrium after this point.
\end{enumerate}

Let $\nu_t := P_t^\ca{C}$ (using censoring scheme notation) be the dynamics of this system. After step 1, we would like to use lemma \ref{wilson-upper-bound-lem} to say that $\nu_t$ is close to the measure $\tilde{\nu}_t$.

\begin{prop}\label{bound-to-smooth-prop}
    For any $\delta, \epsilon > 0$, $N$ sufficiently large, and $t \geq t_1$, we have
    $$ \|\tilde{\nu}_t - \nu_t\| \leq \epsilon/3. $$
\end{prop}

\begin{proof}[Proof of Prop \ref{bound-to-smooth-prop}]
    Following \cite{lacoin-at}, section 5.3, by the censoring, for all $t \leq t_1$ we have that $\tilde{\nu}_t = \tilde{\delta}_1$ where $\delta_1$ is the Dirac mass on the entirely shuffled deck of cards. 

    Now, since we have our censoring, the dynamics up to time $t_1$ is just the product of $R$ independent dynamics of $S_{k,\Delta x_i}$. Thus using lemma \ref{wilson-upper-bound-lem} with $t = t_1$ and $N \geq 10/\delta$:
    \begin{align*}
       \|\nu_t - \tilde{\delta}_1\| &\leq \sum_{i=1}^R \|\nu_t^i-\mu^i\| \leq \sum_{i=1}^R 10k\Delta x_ie^{-t\lambda_{k, \Delta x_i}} \leq 11Nke^{-\log(kN)/3\delta} \leq \epsilon/3.
    \end{align*}
    To see that $\|\nu_t-\tilde{\nu}_t\|$ is decreasing, note lemma 3.4 gives us that $\tilde{\nu}_t$ is just the law of $\sigma_t$ for the dynamics started with initial distribution $\tilde{\delta}_1$, and then coupling with the law $\nu_t$ gives us the desired result.
\end{proof}

Now that $\nu_t$ is close to $\tilde{\nu}_t$, the measure which is constant on semi-skeletons, our next goal is to show that after enough time, $\widehat{\nu}_t$ (that is, the induced measure on semi-skeletons) is close to equilibrium $\widehat{\mu}$. For this purpose, it will help us to first show that the induced measure on skeletons is close to equilibrium a little beforehand.

\begin{prop}\label{skel-upperbound-prop}
    For any $\delta, \epsilon > 0$ and $N$ sufficiently large, we have
    $$ \|\bar{\nu}_{t_2}-\bar{\mu}\| \leq \epsilon/3. $$
\end{prop}

\begin{proof}[Proof of Prop \ref{skel-upperbound-prop}]
    The main difference here is that we show
    $$ \frac{\bar{\sigma}(i,j)}{\sqrt{kN}} \xrightarrow[N\to \infty]{} Z(i,j) \quad \text{ in distribution} $$
    where $Z(i,j)$ is a Gaussian of mean 0 and variance $\frac{i}{R}(1-\frac{i}{R})\frac{j}{R}(1-\frac{j}{R})$. This can be shown by verifying that for a $k$-permutation $\sigma$,
    $$ \mu\p{\bar{\sigma}(i,j) = \ell-\frac{x_iy_j}{N}} = \frac{\binom{y_j}{\ell}\binom{kN-y_j}{kx_i-\ell}}{\binom{kN}{kx_i}}, $$
    and using this to establish a local central limit theorem. All of the other steps needed to prove this proposition follow exactly as in \cite{lacoin-at}, section 5.4.
\end{proof}

Now that the skeleton measure is close to equilibrium, we can censor again and use similar arguments to proposition $\ref{bound-to-smooth-prop}$ to obtain the following at time $t_3$.

\begin{prop}\label{semi-skel-upperbound-prop}
    For any $\delta, \epsilon > 0$ and $N$ sufficiently large, we have
    $$ \|\widehat{\nu}_{t_3}-\widehat{\mu}\| \leq 2\epsilon/3. $$
\end{prop}

\begin{proof}[Proof of Prop \ref{semi-skel-upperbound-prop}]
    This follows exactly as in \cite{lacoin-at}, section 5.5, making the same adaptations as those for the proof of proposition \ref{bound-to-smooth-prop}.
\end{proof}

Given propositions \ref{bound-to-smooth-prop} and \ref{semi-skel-upperbound-prop}, theorem \ref{skn-upper-bound-thm} is immediate from lemma \ref{mix-decompose-lemma} and theorem \ref{censoring-thm}.


\section{Upper Bound for $k$-SEP}

The goal now is to prove the upper bound on the mixing time of the $k$-SEP. That is,

\begin{thm}
    Fix $\epsilon, \delta > 0$. For all $N$ sufficiently large,
    $$ d^{k,N,m}\p{(1+\delta)\frac{1}{2k\pi^2}N^2\log m} \leq \epsilon. $$
\end{thm}

To this end, we construct an alternative coupling to that described in section \ref{grand-coupling-section} and analogous to that described in \cite{lacoin-at}, section 8.1. We use techniques from \cite{lacoin-at} in essentially the same manner, so we leave out some details and refer the interested reader to that paper for parts that work identically there.

Let $\Theta := \{(x,z) \mid x \in [N-1],\  \max(0, kx-Nk+m) \leq z+xm/N \leq \min(kx, m)\}$ (so that $z$ are possible values for $\eta(x)$), and let $\ca{T}^\uparrow$, $\ca{T}^\downarrow$ be two independent Poisson processes with rate $k^2$ indexed by $\Theta$. Lastly, define
$$ p_{a,b} := \frac{b(k-a)}{k^2} \qquad \qquad q_{a,b} := \frac{a(k-b)}{k^2}. $$
Now we describe the coupling.
\begin{itemize}
    \item Suppose $\ca{T}^\uparrow_{(x,z)} = t$. If $\eta_{t^-}^\xi(x) = z$ and $(\gamma_{t^-}^\xi(x), \gamma_{t^-}^\xi(x+1)) = (a, b)$, then with probability $p = p_{a,b}$ we let $\eta_t^\xi(x) = z+1$, and with probability $1-p$ we do nothing.
    \item Suppose $\ca{T}^\downarrow_{(x,z)} = t$. If $\eta_{t^-}^\xi(x) = z$ and $(\gamma_{t^-}^\xi(x), \gamma_{t^-}^\xi(x+1)) = (a, b)$, then with probability $q = q_{a,b}$ we let $\eta_t^\xi(x) = z-1$, and with probability $1-q$ we do nothing.
\end{itemize}
The dynamics obtained is precisely the $k$-SEP and retains the same monotonicity property proved in lemma \ref{grand-monotonicity} for the coupling introduced in section \ref{grand-coupling-section}. Let $\bbp$ be the law of the coupling just introduced. We have for any $\xi$ that $\|P_t^\xi -\mu\| = \|P_t^\xi-P_t^\mu\| \leq \max_{\xi'} \|P_t^\xi-P_t^{\xi'}\|$ by the triangle inequality. Further, by the characterization of total variation distance in terms of optimal couplings and the monotonicity of this coupling, $\|P_t^\xi-P_t^{\xi'}\| \leq \bbp(\eta_t^{\xi} \neq \eta_t^{\xi'}) \leq \bbp(\eta_t^\land \neq \eta_t^\lor)$. Therefore, $\max_\xi \|P_t^\xi-\mu\| \leq \bbp(\eta_t^\land \neq \eta_t^\lor)$ and our task is reduced to bounding the latter term.

We do this by splitting the procedure into three steps. Define
\begin{align*}
    t_1 &:= (1+\delta/4)\frac{N^2}{2k\pi^2}\log m, \\
    t_2 &:= (1+\delta/2)\frac{N^2}{2k\pi^2}\log m, \\
    t_3 &:= (1+\delta)\frac{N^2}{2k\pi^2}\log m.
\end{align*}
The idea will be that after $t_2$, the area between $\eta_t^\land$ and $\eta_t^\lor$ will be smaller than the equilibrium fluctuations. Formally, that is to say
$$ A(t) := \sum_{x=1}^{N-1} (\eta_t^\land-\eta_t^\lor)(x) \ll N\sqrt{m}. $$
Then, between times $t_2$ and $t_3$, we will compare the supermartingale $A(t)$ with a random walk to get convergence to zero with high probability. From this point, the result is straightforward.

\begin{prop}\label{ksep-path-converge-prop}
    $\bbp(\eta_{t_3}^\land \neq \eta_{t_3}^\lor) \leq \epsilon$.
\end{prop}

However, to perform this last step, we need that the distribution of $\eta_t^\land$ (and $\eta_t^\lor$) is close to equilibrium for technical reasons that will appear in the latter half of section \ref{area-decay-section}. This leads us to our next section.

\subsection{Reaching equilibrium}

\begin{thm}
    For all $N$ sufficiently large,
    $$ \lim_{N\to\infty} \|P_{t_2}^\land - \mu\| = \lim_{N\to\infty} \|P_{t_2}^\lor - \mu\| = 0. $$
\end{thm}

\begin{proof}
    The proof is essentially the same as in \cite{lacoin-at}, section 8.2. Thus, we will only give the broad ideas and modifications, and refer the reader to \cite{lacoin-at} for the details. Since the processes are symmetric, we only prove convergence for $\land$.

    Fix $R=\ceil{1/\delta}$. Between time 0 and time $t_1$, we run the standard (non-censored) dynamics. Between time $t_1$ and time $t_2$, we run a censored dynamics, which censors updates at $(x_k, i_k, j_k)$, analogously to section \ref{skn-upper-bound-section}. Denote the law of $\eta^\land_t$ under this dynamics by $\nu_t$. By our censoring inequality, we have $\|P_t-\mu\| \leq \|\nu_t-\mu\|$ for all $t\geq 0$. Since $\eta_t$ essentially corresponds to a semi-skeleton of $S_{k,N}$, we can use the logic behind the proof of proposition \ref{semi-skel-upperbound-prop} and bound 
    $$\|\nu_t-\mu\| \leq \|\bar{\nu}_t-\bar{\mu}\| + \max_\xi \|\nu_t(\cdot \mid \bar{\eta} = \xi) - \mu(\cdot \mid \bar{\eta} = \xi)\|,$$
    our goal being to bound this by arbitrary $\epsilon>0$. The latter term is easily bounded by $\epsilon/2$ (for large $N$) using lemma \ref{wilson-upper-bound-lem} for time $t_2-t_1$, since the dynamics are a product of $R$ independent $k$-SEPs during this time. The prior term requires similar treatment to that done in the proof of proposition \ref{skel-upperbound-prop}.

    Let $\ca{A}_i = \{\eta \mid \bar{\eta}(i) \geq A\sqrt{m}\}$, let $\ca{A} = \bigcap_{i=1}^{R-1} \ca{A}_i$, and let $\ca{B} = \bigcap_{i=1}^{R-1} \ca{A}_i^c$. We can notice that
    $$ \p{\sqrt{\frac{Nk}{m(Nk-m)}}\bar{\eta}(i)}_{i=0,\ldots,R} \xrightarrow[N\to\infty]{} (Y_i)_{i=0,\ldots,R} \quad \text{ in distribution} $$
    where $Y$ is a Gaussian process with covariance function $\E[Y_iY_j1_{i\leq j}] = \frac{i}{R}(1-\frac{j}{R})1_{i\leq j}$, and this is because the analogous process replacing $\bar{\eta}$ with $\eta$ converges to the Brownian bridge. Therefore, for large $A$ there exist $\delta_1, \delta_2 > 0$ small (as small as needed) such that $\mu(\ca{A}) \geq \delta_1$ and $\mu(\ca{B}) \geq 1-\delta_2$. Let $v(\bar{\eta}_t) = \sum_{i=1}^{R-1} \bar{\eta}(i)$ be the volume below the skeleton $\bar{\eta}_t$. By proposition \ref{heat-prop}, we can show that $\nu(v(\bar{\eta})) \leq \delta_1\alpha A\sqrt{m}$ for any fixed $\alpha$ for $N$ sufficiently large. We can use then this to prove that $\nu(A) \leq (1+\alpha)$, which will in turn show that $\|\bar{\nu}-\bar{\mu}\| \leq 2\alpha+\delta_2$. Thus, by setting $\alpha = \epsilon/8$ and letting $\delta_2 \leq \epsilon/4$, we have the theorem.
\end{proof}

\subsection{Area Decay}\label{area-decay-section}

We now prove proposition \ref{ksep-path-converge-prop}. As in the previous section, the proof is essentially the same as in \cite{lacoin-at}, section 8.4. Thus, we again only give the broad ideas and modifications and refer the reader to \cite{lacoin-at} for the details.

We say that $x \in [N-1]$ is an active coordinate at time $t$ if there is some $y \in \{x-1, x, x+1\}$ such that $\eta_t^\land(y) > \eta_t^\lor(y)$. Let $B(t) \subseteq [N-1]$ be the set of active coordinates at time $t$. Given $\eta$ a height function, let $f_x(\eta) = (\gamma(x), \gamma(x+1))$. Define the sets of active points by
\begin{align*}
    C^1_{ab} &= \{(x,z): x\in B(t), \eta_t^\land(x) = z, f_x(\eta_t^\land) = (a,b)\} \\
    C^2_{ab} &= \{(x,z): x\in B(t), \eta_t^\lor(x) = z, f_x(\eta_t^\lor) = (a,b)\},
\end{align*}
and let $C = \bigcup_{a,b} C^1_{ab} \cup C^2_{ab}$. The rate of decrease and increase for $A(t)$ are respectively 
\begin{align*}
    d(t) &= \sum_{a,b} \p{a(k-b)|C^1_{ab}|+b(k-a)|C^2_{ab}|} \\
    u(t) &= \sum_{a,b} \p{a(k-b)|C^2_{ab}|+b(k-a)|C^1_{ab}|}.
\end{align*}
One can observe that $(d-u)(t) \in \{0, k, 2k, \ldots, 2k^2\}$ and thus $A(t)$ is a supermartingale.

Given a sequence of iid exponential variables $(e_n)_{n\geq 0}$ and a Bernoulli sequence with parameters $1/2$, $(V_n)_{n\geq 0}$, we can reconstruct the dynamics of $(\eta^\land_t, \eta^\lor_t)$ as follows.
\begin{itemize}
    \item Updates of nonactive coordinates are performed with appropriate rate independently of $e$ and $V$.
    \item For active coordinates: After the $(n-1)$th update, wait for $e_n/(u(t)+d(t))$ and then
    \begin{itemize}
        \item If $V_n = -1$, choose an active point randomly in $C$ but pick points in $C^1_{ab}$ with probability $a(k-b)/d(t)$ and in $C^2_{ab}$ with probability $b(k-a)/d(t)$. If the active point was $(x,z) = (x,\eta^\land_t(x))$, then subtract 1 from $\eta^\land_t(x)$. If instead $(x,z) = (x,\eta^\lor_t(x))$, then add 1 to $\eta^\lor_t(x)$.
        \item If $V_n = 1$, then with probability $\frac{d-u}{d+u}(t)$, we choose an active point randomly in $C$ and update as we did in the previous bullet point. With probability $\frac{2u}{d+u}(t)$, we pick a point randomly in $C$ but pick points in $C^1_{ab}$ with probability $b(k-a)/u(t)$ and in $C^2_{ab}$ with probability $a(k-b)/u(t)$. We then add 1 if the active point came from $\eta^\land_t$ an subtract 1 if it came from $\eta^\lor_t$.
    \end{itemize}
\end{itemize}

$\eta_t^\land$ and $\eta_t^\lor$ merge after finitely many updates of active coordinates whp, so we let $\ca{N}$ be the last one used. Let $W_n = 1$ if the $n$th update increases the area and $W_n = -1$ if it decreases the area. By construction, $W_n \leq V_n$. Let $\tilde{S}_t$ be a random walk defined by $\tilde{S}_t = A(t_2) + \sum_{n=1}^N W_n$ where $\sum_{n=1}^N e_n \leq t < \sum_{n=1}^{N+1} e_n$. If $t \geq \sum_{n=1}^{\ca{N}} e_n$, then we let $\tilde{S}_t = 0$. Note that

$$ A(t+t_2) = \tilde{S}\p{\int_0^t (d(s)+u(s))ds}. $$

We also define stopping times as follows:
\begin{align*}
    &\tau_i := \min_{t\geq 0} \{\tilde{S}_t \leq m^{1/2-(i+1)\epsilon}N\} &&\tau_\infty := \min_{t\geq 0} \{\tilde{S}_t = 0\} \\
    &\tau'_i := \min_{t\geq 0} \{A(t+t_2) \leq m^{1/2-(i+1)\epsilon}N\} &&\tau'_\infty := \min_{t\geq 0} \{A(t+t_2) = 0\}.
\end{align*}
Because of our time-changing relation between $A$ and $\tilde{S}$, we have that $\tau_{i+1}-\tau_i = \int_{\tau'_i}^{\tau'_{i+1}} (d(t)+u(t))dt$. In particular, a good lower bound on $d(t)+u(t)$ and a good upper bound on $\tau_{i+1}-\tau_i$ would give a good upper bound for $\tau'_{i+1}-\tau'_i$ and thus $\tau'_\infty$. Note that by proposition \ref{heat-prop}, we know that $\tau_2 = \tau'_2 = 0$ whp. Additionally by comparing $\tilde{S}$ with a simple random walk (by replacing $W$ with $V$) we get the following lemma.

\begin{lem}[\cite{lacoin-at}, Lemma 8.6]\label{walk-stopping-time-lem}
    If $\epsilon < \delta/100$. Then whp $\tau_\infty - \tau_{\ceil{1/(2\epsilon)}+1} \leq N^2$ and for all $i \in \{2,\ldots, \ceil{1/2\epsilon}\},$
    $$ \tau_{i+1}-\tau_i \leq m^{1-(2i+1)\epsilon}N^2. $$
\end{lem}

This covers the latter part of our strategy. Thus, we're left with providing a good lower bound on $d(t)+u(t)$. In order to do this, we define a set of configurations that present a lot of active coordinates (or rather, those that don't). Let $\ca{H}_1 = \{\eta \mid \max_x |\eta(x)| \geq \sqrt{m}\log m\}$ and let $\ca{H}_2 = \{\eta \mid \gamma_\eta([x, x+2(N/m)(\log m)^2]) = \{0\} \text{ or } \{k\}\}$ where $\gamma_\eta$ is the configuration associated to $\eta$. Then we say $\ca{H} = \ca{H}_1 \cup \ca{H}_2$ consists of the ``bad" configurations. Our crucial lemma is thus that we avoid bad configurations most of the time after $t_2$.

\begin{lem}[\cite{lacoin-at}, Lemma 8.8]\label{avoid-bad-paths-lem}
    $\lim_{N\to\infty} \mu(\ca{H}) = 0$ and thus
    $$ \lim_{N\to\infty} \bbp\s{\p{\int_{t_2}^{t_2+\ceil{1/2\epsilon}N^2} 1_{\eta_t^\land \in \ca{H} \text{ or } \eta_t^\lor \in \ca{H}}dt} \geq N^2/2} = 0. $$
\end{lem}

All that is left to show is that configurations that are not bad do indeed present a good lower bound on $d(t)+u(t)$. This is given by the following lemma.

\begin{lem}[\cite{lacoin-at}, Lemma 8.9]
    If $2 \leq i \leq \ceil{1/2\epsilon}$ and $t< \tau'_{i+1}$ and $\eta_t^\land, \eta_t^\lor \notin \ca{H}$, then
    $$ (d+u)(t) \geq \frac{m^{1-(i+2)\epsilon}}{8(\log m)^2} $$
\end{lem}

Using the previous 3 lemmas, we can show that $\bbp[\exists i \mid \tau'_{i+1}-\tau'_i \geq N^2] \to 0$ and $\tau_\infty - \tau_{\ceil{1/(2\epsilon)}+1} \leq N^2$ whp and thus $\tau'_\infty \leq \ceil{1/(2\epsilon)}N^2$ whp.

\begin{proof}[Proof of Proposition \ref{ksep-path-converge-prop}]
    If $t_3-t_2 \geq \ceil{1/(2\epsilon)}N^2$, N is sufficiently large, and $\epsilon < \delta/100$, then $\bbp(\eta_{t_3}^\land \neq \eta_{t_3}^\lor) = \bbp(\tau'_\infty > t_3-t_2) \leq \epsilon$.
\end{proof}


\section{Acknowledgements}
The author thanks Evita Nestoridi for introducing the problem to him and providing valuable guidance and support throughout the research process. This research was partially funded by the Graduate Fellowships for STEM Diversity.


\printbibliography

\end{document}